\documentclass[11pt]{amsart}
\usepackage{amsthm, amsfonts, amsxtra, amssymb, amscd}

\usepackage{amsfonts,amsmath,amssymb,amscd,amsthm}
\usepackage[english]{babel} 
\usepackage[all]{xy}
\usepackage[backref, colorlinks, linktocpage, citecolor = blue, linkcolor = blue]{hyperref}

\newcommand{\ii}{{\rm i}}
\newtheorem*{lemma*}{Lemma}

\newtheorem*{theorem*}{Theorem}
\newtheorem{theorem}[subsection]{Theorem}
\newtheorem*{proposition*}{Proposition}
\newtheorem{proposition}[subsection]{Proposition}
\newtheorem*{corollary*}{Corollary}

\theoremstyle{definition}
\newtheorem*{definition*}{Definition}
\newtheorem{definition}[subsection]{Definition}
\newtheorem*{example*}{Example}

\theoremstyle{remark}
\newtheorem*{remark*}{Remark}
\newtheorem{remark}[subsection]{Remark}

\newcommand{\C}{\mathbb C}

\renewcommand{\phi}{\varphi}

\newcommand{\be}{\begin{enumerate}}
\newcommand{\ee}{\end{enumerate}}

\title{Some special bases of the 2--swap algebras.}
\author{Claudio Procesi}
\begin{document}\address{Dipartimento di Matematica, G. Castelnuovo,
Universit\`a di Roma La Sapienza, piazzale A. Moro,  00185,
Roma, Italia}
 \thanks{I wish to thank Felix Huber for pointing out the problem and some of the literature.} 

 \begin{abstract}
We study the algebra $\Sigma_n$  induced by the action of the symmetric group $S_n$ on $V^{\otimes n}$ when $\dim V=2$. 

Our main result is that the space of symmetric elements  of  $\Sigma_n$ is linearly spanned  by the involutions of $S_n$.
\end{abstract}

\email{
     procesi@mat.uniroma1.it}
 \bigskip
\maketitle
\section{Introduction}  Let $V$ be a vector space of finite dimension $d$ over some field $F$  (in our computations it is convenient to take $F=\mathbb Q$ the rational numbers or $\C$ the complex numbers).  

In the classical theory of Schur--Weyl  a major role is played by the action of the symmetric group $S_n$ on $n$ elements on the $n^{th}$ tensor power $V^{\otimes n}$ by exchanging the tensor factors. The algebra of operators on $V^{\otimes n}$, generated  by these permutations will be denoted  by $\Sigma_n(d)$  and called {\em a  $d$--swap algebra}.  It is the algebra formed by the elements which commute with the diagonal action of $GL(V)$ or, if $V$ is a Hilbert space, by the corresponding unitary group.

          The name comes from the use, in the physics literature, to call {\em swap}  the exchange operator $(1,2)$ on $V^{\otimes 2}$. 
          
           In the literature on quantum information theory  the states lying in  $\Sigma_n(d)$ are called {\em Werner states}   and  widely used as source of examples, due to fundamental work of the physicist R. F. Werner \cite{RW}.   
          
          See for instance \cite{TR},\  \cite{WC},\  \cite{HK},\ \cite{EW} for applications to separability, entanglements or the quantum max-cut problem.

\medskip

A  classical theorem  states   that the corresponding algebra homomorphism  $F[S_n]\to \Sigma_n(d)\subset End(V^{\otimes n})$ is injective if and only if  $\dim V\geq n$. \smallskip

When $d=\dim V< n$ the kernel of this map is the two sided ideal of $F[S_n]$ generated by the antisymmetrizer
$$A_{d+1}:=\sum_{\sigma\in S_{d+1}}\epsilon_\sigma \sigma,\quad \epsilon_\sigma \ \text{the sign of the permutation}. $$
If $F$ has characteristic 0 the algebra $F[S_n]$ decomposes as direct sum of matrix algebras indexed by partitions, corresponding to the irreducible representations of $S_n$. As for  $\Sigma_n(d)$  only the blocks relative to partitions of height $\leq d$ survive.\medskip

In the case $d=\dim V< n$ an interesting problem is to describe a basis of  $\Sigma_n(d)$ formed by permutations.  In fact in the physics literature there are several examples of Hamiltonians lying in  $\Sigma_n(d)$.  Thus it may be convenient to express such Hamiltonian in a given special basis,\smallskip

 Moreover  in $\Sigma_n(d)$  we have the involution $g\mapsto g^{-1}$, if $V$ is a Hilbert space this coincides with adjuction.    So we would like also   to describe a basis of  $\Sigma_n(d)^+$ the subspace of symmetric elements made by permutations and also a basis of self adjoint operators.\medskip

In \cite{P}  I have proved that a possible basis of  $\Sigma_n(d)$ is formed by the permutations which are $d+1$--good.   

By definition a   permutation $g$ is $d+1$--good if and only if it does not contain a decreasing subsequence of length  $d+1$. \smallskip

By a beautiful Theorem of Schensted  this is equivalent to the fact that the pair of tableaux  associated to  $g$ is of height $\leq d$.  

This of course, by classical theory,  is  exactly the dimension of  $\Sigma_n(d)$ so it is enough to prove that such permutations span  $\Sigma_n(d)$ and this, in  \cite{P}  is done by a {\em straightening algorithm}  deduced from the relations.  

Since the pair of tableaux  associated to $g^{-1}$ by the Robinson--Schensted correspondence is obtained by exchanging that  associated to $g $, it follows  that if $g$ is $d+1$--good so is   $g^{-1}$, and this gives  also a basis  $g+g^{-1}$, where $g$ is $d+1$--good,  for the symmetric elements. 

We  have also  the basis  $g-g^{-1}$ for the antisymmetric elements, from which we have a basis over the real numbers for self adjoint operators   given by $g+g^{-1}$and $i(g-g^{-1})$ , where $g$ is $d+1$--good.\medskip

 On the other hand,  specially for $d=2$,  one may want to find a basis  formed by simpler type of elements.  For this discussion the simplest elements are the elements of order 2 (called {\em involutions}) which are permutations with cycles only of order 2,1 and eigenvalues only $\pm 1$.  \smallskip
 
 We call $\Sigma_n(2)$ the {\em $n$--swap algebra} and denote it simply $\Sigma_n.$  It is known that $\dim \Sigma_n=C_n$ the $n^{th}$ Catalan number, see \S\ref{Ca} for a simple proof.   
 
 The list of the first 10  Catalan numbers is
 $${1, 2, 5, 14, 42, 132, 429, 1430, 4862, 16796}$$

\begin{definition}\label{spe}
The set $\mathcal S$ of {\em special permutations}   is formed by the involutions and also by the permutations with cycles only of order 2,1 plus a single cycle of order 3.
\end{definition}
The 3 cycle can be further normalised to be increasing.\medskip

Our main Theorem is the following
\begin{theorem}\label{main}\begin{enumerate}\item For each $n$ the algebra  $\Sigma_n$ has a basis formed by special elements.

\item $\Sigma_n^+$ has a basis  over $\mathbb C$ formed by involutions.
 \item The space  of real and symmetric elements  has a basis  over $\mathbb R$ formed by involutions.\end{enumerate}

\end{theorem} 
Notice that items (2) and (3)  are equivalent  and follow  from (1).  

 In fact the involutions are symmetric and if a permutation  is of the form   $g=ab$  with one 3 cycle  $a$ and the rest $b$ is a product  2 or 1 cycles   its symmetrization is $(g+g^{-1})=(a+a^{-1})b$.  

If $a$ is a 3 cycle  in the algebra  $\Sigma_3$, by relation  \eqref{uno}, we have that  $a+a^{-1}$  is the sum of  -1 and   3  transpositions. The claim follows. 

In the same way   $(g-g^{-1})=(a-a^{-1})b$  gives bases for antisymmetric elements, and together for self adjoint operators.\medskip

  The dimensions  of the real symmetric elements are, from  $n=1$   to   $n=10$  
$${1, 2, 4, 10, 26, 76, 232, 750, 2494, 8524,\cdots}$$
(see  
 The On-Line Encyclopedia of Integer Sequences 
A007123 for many interesting informations on this sequence).

On the other hand the number $I(n)$ of involutions in $S_n$ from  $n=1$   to   $n=10$ is
$$I(n)={1, 2, 4, 10, 26, 76, 232, 764, 2620, 9496,\cdots}$$  which is also equal  (by the Robinson--Schensted correspondence) to the number of standard Young tableaux with $n$ cells (O.E.I.S A000085).

So a curious fact is that these two sequences coincide up to $n=7$. \smallskip

We have thus that the involutions are a basis of the real symmetric elements  for $n\leq 7$  and after that  they have linear relations.  It would be interesting to understand these relations. 

The antisymmetric elements   are spanned by  elements  of type  $ab-a^{-1}b$  with $a$ a 3 cycle, which we can assume to be in increasing order, and $b$ an involution.  There are $\binom n3$ such 3 cycles in $S_n$ and so  $\binom n3I(n-3)$ such elements.
$$\binom n3I(n-3)=1,4,20,60,350,\ n=3,\cdots,7$$ so there appear relations already for $n=5$   where the dimension of $\Sigma_5$ is 42, while the number of normalised special elements is 46.\medskip

\begin{remark}\label{pat}
The set of special elements has the following compatibility with the {\em partial traces}\quad  $ \mathtt t_i:End(V)^{\otimes n}\to  End(V)^{\otimes n-1},$
$$\mathtt t_i:x_1\otimes\cdots  x_i\otimes \cdots x_n\mapsto tr(x_i)x_1\otimes\cdots  x_{i-1}\otimes x_{i+1}\otimes\cdots \otimes x_n.$$
In fact applied to a permutation  decomposed into cycles $\mathtt t_i$ removes $i$  from the cycle, and so a special element is mapped to a special element,  and moreover, in case $i$ is a fixed element multiplies by the dimension of the space, in our case 2.\smallskip

Notice that  instead  the partial trace of a 3--good element may be 3--bad as for instance $$\mathtt t_4(\{3,4,1,2\}) =\{3,2,1\}$$  where by $\{3,4,1,2\}$ we mean the permutation as string and not as cycle, in cycle form $\{3,4,1,2\}=(1,3)(2,4)\mapsto (1,3)(2 )=\{3,2,1\}$.\end{remark}
\begin{remark}\label{bass}
Given a basis $e_1,e_2,e_3,e_4$  for  the space $End(V)$  of linear operators on $V$  one has the {\em dual basis} $f_i,\ i=1,\cdots,4$ for the trace form $tr(ab)$. That is  4  operators satisfying $tr(e_if_j)=\delta^i_j$. Then the operator $(1,2):V\otimes V\to V\otimes V$ can be written as $$(1,2)=e_1\otimes f_1+e_2\otimes f_2+e_3\otimes f_3+e_4\otimes f_4.$$  Any involution  being product of elements $(i,j)$ can the be expressed using this Formula in term of the basis.

\end{remark} We have some freedom in the choice of the basis. The   most common is the basis by matrix units $e_{i,j}$    in which
$$(1,2)=e_{1,1}\otimes e_{1,1}+e_{1,2}\otimes e_{2,1}+e_{2,1}\otimes e_{1,2}+e_{2,2}\otimes e_{2,2}. $$
In particular in Physics are widely used the {\em Pauli matrices}.
$$\sigma_0:=\begin{vmatrix}
1 &0\\ 0&1 
\end{vmatrix},\sigma_x:=\begin{vmatrix}
0&1\\1&0
\end{vmatrix},\sigma_y:=\begin{vmatrix}
0&-\ii\\\ii&0
\end{vmatrix},\sigma_z:=\begin{vmatrix}
1 &0\\0  &-1 \end{vmatrix} $$
$$\frac 12\begin{vmatrix}
1+a&b-\ii c\\b+\ii c&1-a
\end{vmatrix}=\frac 12(I+b\sigma_x+c\sigma_y+a\sigma_z). $$ They are equal to the dual basis up to a scaling by $\frac 12$  so that:
$$(1,2)=\frac12(\sigma_0\otimes \sigma_0+\sigma_x\otimes \sigma_x +\sigma_y\otimes \sigma_y+\sigma_z\otimes \sigma_z). $$
\bigskip

The proof of this Theorem is algorithmic. We give an algorithm which, given as input a permutation $\sigma\in S_n$  produces a linear combination of   elements in $S$ which in $\Sigma_n$ equals to $\sigma$.  
\section{The algorithm} Usually we write the permutations in their cycle structure.
Let us start with the basic antisymmetrizer which vanishes in $\Sigma_3$. 
$$A=(1,2,3)+(1,3,2)-(1,2)-(1,3)-(2,3)+1$$
\begin{equation}\label{uno}
(1,2,3)+(1,3,2)=(1,2)+(1,3)+(2,3)-1.
\end{equation}  First remark  that in $S_3$  all permutations are special, moreover
 \begin{equation}\label{ra}
(1,3,2)=-(1,2,3)+(1,2)+(1,3)+(2,3)-1  .
\end{equation} so a 3--cycle can be normalised. In $S_4$  we have the 4-cycles which are not special and we have to  write them as linear combination of special permutations in $\Sigma_4$.  Notice that  if we can do this for a single cycle  we can do it for all cycles, since permutations of the same cycle structure are conjugate and clearly the space spanned by special permutations is closed under conjugation.\medskip

For any $n>3$  we have the natural embedding of $S_3$ in $S_n$ as the permutations on the first 3 elements. This  induces an embedding of $A$ in the algebra of the symmetric group of $S_n$ which we denote by $An$.  This element   vanishes in the swap algebra $\Sigma_n$. and in $\C[S_n]$  generates the ideal of relations for $\Sigma_n$. 
Thus in $\Sigma_4$ we have the vanishing of 
$$(2,4)A4= (4, 2, 3, 1)+(3, 4, 2, 1)-(4, 2, 1)-(3, 1) (4, 
   2)-(3, 4, 2)+(2,4) $$   
$$(3,4)A4(2,4)=(2, 3, 1)+(4, 1)(2,3) - (2, 3, 4, 1)- (4, 2, 3, 
   1)- (2,3) +(3, 4, 2)$$
     $$A4(3,4)=(2, 3, 4, 1)+(3, 4, 2, 1)- (1,2) ( 3,4)- (3, 4, 
   1)-(3, 4, 2)+( 3,4). $$  
     We then have, in $\Sigma_4$,  $0=A4(3,4)-(3,4)A4(2,4)$  that is     

     \begin{equation}\label{pp}
-2(2, 3, 4, 1)=(3, 4, 2, 1)+(4, 2, 3, 
   1)- (1,2) ( 3,4)- (3, 4, 
   1)\end{equation}$$-2(3, 4, 2)+( 3,4)-(2, 3, 1)-(4, 1)(2,3) +(2,3) $$
From the vanishing of $(2,4)A4$  we deduce:
   $$ (4, 2, 3, 1)+(3, 4, 2, 1)=(4, 2, 1)+(3, 1) (4, 
   2)+(3, 4, 2)-(2,4) $$
   Substituting in \eqref{pp}  we deduce:

\begin{equation}\label{q}2(1,2, 3, 4)=\end{equation} 
$$\!\! \!\! \!\! \!\!  \!\! \!\!  (1,4 )(2,3) +(1,2) ( 3,4)-(1,3 ) (4, 
   2) -(1,4, 2 )+(1,3, 4)+(2,3, 4)+(1,2, 3)+(2,4)- ( 3,4)- (2,3) .$$

 %
%
%
%
%
%
%
%
%

%
%

The term  $(4,2,1)$ is not normalised, but it can be rewritten using Formula \eqref{ra}.                 
   $$-(4, 2, 1)=(1,2,4)-(1,2)-(1,4)-(2,4)+1     $$
\begin{equation}\label{q}2(1,2, 3, 4 )=\end{equation} 
$$\!\! \!\! \!\! \!\!  \!\! \!\!  (1,4)(2,3) +(1,2) ( 3,4)-(1,3) (4, 
   2) +(1,2,4)+(1,3, 4 )+(2,3, 4)+(1,2, 3 ) -(1,2)-(1,4)- ( 3,4)- (2,3)  +1.$$
 Since all 4 cycles are conjugate we deduce that  statement (1) is true for $S_4$.\smallskip
 
 Now notice the following  general fact: consider two cycles  $(a,A),\  (a,B)$  of lengths $h,k$ respectively where $A$ and $B$ are strings of integers of lengths $h-1,k-1$ respectively and disjoint. Then their product is the cycle of length $h+k-1$:
 \begin{equation}\label{id}
( a,B)(a,A)= (a,A,B),\quad e.g.\  (1,2,3)(1,5,4,6)=(1,5,4,6,2,3).
\end{equation}
 Thus take a cycle of length $p>4$  and, up to conjugacy we may take  \begin{equation}\label{ilc}
\mathtt c_p:=(1,2,3,4,5,\ldots,p)=(1,5,\ldots,p)(1,2,3,4) .
\end{equation} In $\Sigma_p$ we have thus that  $2\mathtt c_p$  equals  $(1,5,\ldots,p)$ times the expression of Formula  \eqref{q}. 
 
  But then applying again  Formula \eqref{id} we see that the resulting formula is a sum of permutations on $p$ elements which are {\bf not} full cycles.  
  
  By iterating then the operation on the cycles  of length $\ell$ with $4\leq \ell\leq p-1$ we have a preliminary.
 \begin{proposition}\label{se}
The cycle  $\mathtt c_p$ (formula \eqref{ilc})  is a linear combination in $\Sigma_p$  of elements which contain only cycles of length 1,2,3.

Hence $\Sigma_n$  is spanned by  permutations  which contain only cycles of length 1,2,3. \end{proposition}
{\bf Example}\quad For $p=5,6$  we have for $2\mathtt c_p$ the formula  obtained from  Formula  \eqref{q}:
    \begin{align*} &2(1,2,3,4,5)\stackrel{\eqref{ilc}}=2(1,5)(1,2,3,4)=\\
&(  1,4,5)(2,3) +( 1,2, ,5) ( 3,4)-( 1,3,5) (4, 
   2) -(1,4, 2,  5)+(1,3, 4,   5)\\&+(3, 4, 2)(1,5)+(1,2, 3,   5)+(2,4)(1,5)- ( 3,4)(1,5)- (2,3) (1,5) 
\end{align*}
    
In the previous formula appear three  4-cycles, for which we can apply  Formula   \eqref{q} (see in the appendix, the expanded Formula \eqref{15}).

Notice that the final Formula  must be  invariant under conjugation by powers of the cycle, but this only up to the relations  in $\Sigma_n$.   
      \begin{equation} \label{sei} 2(1,2,3,4,5,6) \stackrel{\eqref{ilc}}=2(1,5,6)(1,2,3,4)=\end{equation}
  $$(  1,4,5,6)(2,3) +(  1,2,5,6) ( 3,4)-(  1,3,5,6) (4, 
   2) -(1,4, 2,  5,6)+(1,3, 4,   5,6)$$$$+(3, 4, 2)(1,5,6)+(1,2, 3,  5,6)+(2,4)(1,5,6)- ( 3,4)(1,5,6)- (2,3) (1,5,6) .
$$

Of course in the previous formulas appear 4-cycles, for which we can apply  Formula   \eqref{q}, and then 5-cycles, for which we can apply the final formula developed before.  Notice now that in Formula  \eqref{sei}  all terms are either special or can be expanded into a linear combination  of special elements, using the formulas of 4 and 5--cycles, except the term   $(3, 4, 2)(1,5,6)$.  \medskip

In order to  prove  Theorem \ref{main}  using Proposition \ref{se} it is enough to prove that, in $S_6$,  a permutation of type $3,3$  can be developed as linear combination of special elements, since then we apply recursively this to a product of $k$ disjoint 3-cycles.  If $k$ is even we replace them all and if odd we remain with only one  3-cycle which can be normalized if necessary using Formula \eqref{ra}.\smallskip

The computation in $S_6$ in principle is similar to that in $S_4$  but now we have to handle a priori many more relations and I had to be assisted by the software "Mathematica" in order to discover the needed relations.  

Let me sketch what I did.
 \subsection{The computation}

  A set of relations for $\Sigma_6$  can be obtained from the antisymmetrizer $A6$  by multiplication to the left and right by the 720 permutations.  Actually it is not necessary  to use all permutations since there are $36\times 6$  pairs which stabilyze $A6$ up to sign
  

Finally using these reductions we have 2400  relations each a sum of 6 permutations, of which 3 even and 3 odd.

  Each  6--cycle $c$,  appearing in these relations, needs to be developed by using the appropriate conjugate of Formula \eqref{sei} by the permutation which has as string the same form of the cycle $c$  and which conjugates the standard 6 cycle into $c$.

  So  a 6 cycle  is replaced, using a conjugate  of  formula \eqref{sei},  by a permutation of type $3,3$  plus a sum of special terms.  In this  way we obtain 2400 relations  which, by inspection contain either  0, 2 or 3 permutations of type $3,3$ and with the remaining terms special.\smallskip
  
    The  ones with 0, 2   permutations of type $3,3$ are  linear combinations of special   permutations  and cannot be used. \smallskip
  
  Remain 360  relations containing 3 permutations of type $3,3$,   arising from relations with   two  6--cycles   and one  permutations of type $3,3$.
\begin{remark}\label{red} There are  40  permutations of type $(3,3 )$  but using  Formula \eqref{uno}  
 we can  normalise these elements. 
 
 If a 3 cycle  $(a,b,c)$ is not strictly increasing (up to cyclic equivalence) we can replace it  by a strictly increasing    cycle    introducing a sign and adding   some special  permutations applying a conjugate of formula \eqref{uno}.   
 
 We then are reduced to 10  normalised  permutations of type $(3,3 )$.  
 
 We have several relations involving   these normalised permutations plus special elements and we have to eliminate in one relation  all  permutations of type $(3,3 )$ except  one, thus obtaining the desired relation.

\end{remark}    
  \subsection{The useful relations\label{ur}}
  Surprisingly  in order to obtain the desired relation  only the following 2 suffice:

\scriptsize  
$$\!\!\!\!\!\!\!\!\!\!\!\!\!\!\!\!\!\!\ 
\begin{array}{cc}
(5,6,1)( 3,4) \text{A6} = \\\left(
1, 2 , 4 , 3 , 5 , 6   \right)\!+\!\left(
 1,4 , 3 , 2 , 5 , 6    
\right)-(2,5,6,1)( 3,4) -(4,3,5,
   6,1)- (
 5 , 6 , 1)(
 4 , 3 , 2 )\!+\!(5,6,1)( 3,4)  
    \\\\ (6,1)(4,5,3)  \text{A6}=\\
\left(
 1,2 , 4 , 5 , 3 , 6  \right) + \left( 1,4 , 5 , 3 , 2 , 6 \right)-\! (
 2 , 6 , 1)(
 4, 5 , 3) -(4,5,3,6,1) -(6,1)
   (4,5,3,2)  + (6,1)(4,5,3)   
   \\
    \end{array}
 $$\normalsize

\normalsize

In Formula \eqref{sei}  the contribution  to the expansion of $2(1,2,3,4,5,6)$ of an element of type $3,3$ is   
   $+(3, 4, 2)(1,5,6)$.  
   
   Therefore the contributions of type $3,3$  of the 4 cycles of length 6 appearing in the previous Formulas are
   obtained by conjugating $ (3, 4, 2)(1,5,6)$  with the permutation which has as string the same form of the cycle.       $$\!\!\!\!\!\!\!\!\!\!\!\!\!\!\!\!\!\!\ 
\begin{array}{cc}
2\left(
 2 , 4 , 3 , 5 , 6 , 1 \right)= (3, 5, 4)(2,6,1)+\cdots ,\quad  2\left(
 4 , 3 , 2 , 5 , 6 , 1  
\right)=  (2, 5, 3)(4,6,1) +\cdots  \\
2\left(
 2 , 4 , 5 , 3 , 6 , 1\right)\!\!=(5, 3, 4)(2,6, 1)+\cdots ,\quad  2\left( 4 , 5 , 3 , 2 , 6 , 1\right)
=(3, 2, 5)(4,6,1)+\cdots \\
\end{array}
 $$

%

 \smallskip

So, by Remark \ref{red},  the previous elements  can be written in $\Sigma_6$ as
   $$\!\!\!\!\!\!\!\!\!\!\!\!\!\!\!\!\!\!\ 
\begin{array}{cc}
2\left(
 2 , 4 , 3 , 5 , 6 , 1 \right)= -(3, 4, 5)(1,2,6)+\cdots ,\quad  2\left(
 4 , 3 , 2 , 5 , 6 , 1  
\right)= - (2, 3, 5)(1,4,6) +\cdots \\
2\left(
 2 , 4 , 5 , 3 , 6 , 1\right)\!\!=(  3, 4,5)(1,2,6)+\cdots ,\quad  2\left( 4 , 5 , 3 , 2 , 6 , 1\right)
=-(2, 3, 5)(1,4,6)+\cdots  
\end{array}
 $$
where the $\cdots$ represent special elements.\smallskip

Therefore the previous 2 relations  multiplied by 2  are of the form

 \begin{equation}\label{prf}
\begin{array}{cc}
  -(3, 4, 5)(1,2,6)- (2, 3, 5)(1,4,6) +2(2,3,4)(1,5,6)+\cdots  \\\\
\quad (  3, 4,5)(1,2,6)  -(2, 3, 5)(1,4,6)- 2(3,4,5)(1,2,6)+\cdots \\
\end{array}
\end{equation} 
Subtracting the second from the first  one has the desired Formula:
$$ 0=2(2,3,4)(1,5,6)+\cdots   $$  
a relation with a single permutation $2(2,3,4)(1,5,6)$ of type $(3,3)$  and the remaining elements special. 

This gives the desired expression, which is explicited in the Appendix.

\section{Comments}\subsection{The Formula $\dim \Sigma_n=C_n$\label{Ca}}\qquad \smallskip

The $n^{th}$  Catalan number is $\frac 1{n+1}\binom{2n}n$. By the {\em hook Formula}  it is easily seen that this is the dimension of the irreducible representation of $S_{2n}$  associated to a Young diagram with two rows  of length $n$.  This in turn appears  as  the isotypic component of the $SL(2)$ invariants in $V^{\otimes 2n},\ \dim V=2$.  

By identifying $V$ with $V^*$  as $SL(2)$ representations we have an $SL(2)$ linear isomorphism between  $ End  (V^{\otimes  n})= V^{\otimes  n}\otimes V^{*\otimes  n}  $ and  $V^{\otimes 2n}$ which induces a linear isomorphism with the respective invariants $$\Sigma_n=End_{SL(V)}  V^{\otimes  n}\simeq (V^{\otimes 2n})^{SL(V)} .$$ 
\subsection{Several bases}
Several bases  formed by special elements can be obtained from Theorem \ref{main}, so a main problem is to describe    a {\em best one} by combinatorial means. The main advantage of  the special elements is that their eigenvalues  are $\pm 1$  and the  two  primitive 3-roots of 1.  They also are {\em local}  in the sense that involve only 2  tensor factors at a time, and at most   a single  instance of 3 tensor factors.

One should compare the complexity  of   the algorithm to express a permutation as linear combination of special ones  to that of   the algorithm to express a permutation as linear combination of 3-good ones.\medskip

Some remarks on the algorithm  to express a permutation as linear combination of 3-good ones.

\begin{enumerate}\item Permutations are ordered lexicographically. 
 
\item 
The 3-good permutations on $n$  elements are  in number $C_n$  the \
$n^{th}$  Catalan number.
\item  The last 3-good permutation is $n,1,2,3,\cdots,n-1$.
%
 \end{enumerate}

\medskip

The algorithm takes a permutation $\sigma$ and checks   recursively if there is a string of 3 elements decreasing. If there is not one the  permutation is 3-good. Otherwise as soon as one encounters one such sequence, by applying the antisymmetrizer on these elements one obtains  that  $\sigma$ is equivalent to a sum with signs of 5  permutations which are lexicographically less than $\sigma$.  This means that if we have already developed the previous permutations as linear combination of 3-good ones we immediately obtain the developments for $\sigma$. Notice that in this development the coefficients are all integers.
\subsection{$d\geq3$}

One may ask the same question for  $\Sigma_n(d)$ and $d\geq3$. The first problem is:

Determine the minimum $m=m(d)$ so that   $\Sigma_{m+1}(d)$ is spanned by the permutations which are NOT $m+1$--cycles.

This number $m$  has also other interesting interpretations (see  \cite{agpr} page 331  for the interesting history  of this question).

 The same $m$ is the maximum degree of the generators of invariants  of  $d\times d$ matrices.

It is also the minimum  degree for which, given an associative algebra $R$ over a field of characteristic 0, in which every element $x$ satisfies  $x^d=0$  one has $R^{m(d)}=0$.\smallskip

The known estimates for $m(d)$   are     the lower bound  $m(d)\geq \binom{d+1}2$  due to Kuzmin,  see \cite{Kuz} or  \cite{DF} and the upper bound  $m(d)\leq d^2$ due to Razmyslov,  see \cite{Raz} or \cite{agpr}.  Kuzmin conjectures that $m(d)= \binom{d+1}2$ which has been verified only for  $d\leq 4$.

 \subsection{ Transpositions}
Let us remark a simple fact
\begin{proposition}\label{tra}
The identity plus all transpositions give linearly independent operators in all $\Sigma_n(d)$.
\end{proposition}
\begin{proof}  Of course it is enough to prove this when $d=2$, we do it by induction.   Assume we have a relation
$$0=a \cdot I_n
+\sum_{i<j}a_{i,j}(i,j)$$ with $I_n$ the identity   on $V^{\otimes n}$.
Apply the partial trace $\mathtt t_1=\mathtt t$ on the first factor as in \cite{P1}  
$$\mathtt t: X_1\otimes X_2\otimes \ldots \otimes X_{n-1}\otimes X_n \mapsto tr( X_1) X_2 \otimes \ldots \otimes X_{n-1 }\otimes X_n.$$ As proved in that paper if $S_{n-1}$  is the subgroup of $S_n$ fixing 1 we have
$$\mathtt t(\sigma)=2\sigma,\ \forall \sigma\in S_{n-1}\quad  \mathtt t(\tau (1,i))=\tau,\  \forall \tau\in S_{n-1}$$
So we have in $V^{\otimes n-1}$.
$$0=b\cdot I_{n-1}
+2\sum_{1<i<j}a_{i,j}(i,j),\quad b=2a+\sum_{1<j}a_{1,j}.$$
By induction we have  $b=0,\ a_{i,j}=0,\ \forall 1<i<j$. 

So the relation is, setting  $a_j:=a_{1,j}$, among  $I_n$ and the $n-1$  transpositions  $(1,i),\ i=2,n$.
Apply the partial trace $\mathtt t_n$ on the last factor obtaining in $V^{\otimes n-1}$ that $ a_{1,j}=0,\  \forall j<n$ and finally  a relation among $I_n$ and $(1,n)$  which does not exist.

\end{proof}

\subsection{Appendix  {\em explicit formulas}}
In the formula for $p=5$  we multiply by 2 
  \begin{align*} &4(1,2,3,4,5)\stackrel{\eqref{ilc}}=4(1,5)(1,2,3,4)=\\
&2(  1,4,5)(2,3) +2( 1,2, 5) ( 3,4)-2( 1,3,5) (4, 
   2) -2(1,4, 2,  5)+2(1,3, 4,   5)\\&+2(3, 4, 2)(1,5)+2(1,2, 3,   5)+2(2,4)(1,5)- 2( 3,4)(1,5)-2 (2,3) (1,5) 
\end{align*}
  and develop the 4--cycles we obtain

   \scriptsize$$-2(1,4, 2,  5)+2(1,3, 4,   5) +2(1,2, 3,   5)=$$
$$ (1,2) ( 3,5)+(1,2 ) (4,5)-(1,3 ) (2,5 
    ) +(1,3) ( 4,5)-(1,4) ( 2,5)-(1,4 ) (3,5)- (1,5 )(2,4)  + (1,5)(3,4)  + (1,5 )(2,3) $$$$  +(2,4, 5)+ (1,2,4  )+(1,  3, 5)+(3,4, 5)+(1,3, 4) +(1,3, 5)+(2,3, 5)+(1,2, 3)$$$$  - (2,3) -2(4,5)    - (3,4)   - 2(1,2   ) - ( 2,4  )-(1,  3 ) -(   3, 5)  -(1,  5) +3  .$$

%
\normalsize
   So the formula
   
   \scriptsize

  \begin{equation}\label{15}
4(1,2,3,4,5)\stackrel{\eqref{ilc}}=4(1,5)(1,2,3,4)=
\end{equation}
   $$ (1,2) ( 3,5)+(1,2 ) (4,5)-(1,3 ) (2,5 
    ) +(1,3) ( 4,5)-(1,4) ( 2,5)-(1,4 ) (3,5)+ (1,5 )(2,4)  - (1,5)(3,4)  -  (2,3) (1,5)    $$$$  +(2,4, 5)+ (1,2,4  )+(1,  3, 5)+(3,4, 5)+(1,3, 4) +(1,3, 5)+(2,3, 5)+(1,2, 3)$$$$  - (2,3) -2(4,5)    - (3,4)   - 2(1,2   ) - ( 2,4  )-(1,  3 ) -(   3, 5)  -(1,  5) +3   $$
$$+2(  1,4,5)(2,3) +2( 1,2, 5) ( 3,4)-2( 1,3,5) (4, 
   2)  +2(2,3, 4 )(1,5)  
$$

%
%
%
%

   \normalsize\smallskip

    The  formula \eqref{sei}  for $p=6$
     \begin{equation}  8(1,2,3,4,5,6) \stackrel{\eqref{ilc}}=8(1,5,6)(1,2,3,4)=\end{equation}
  $$4(  1,4,5,6)(2,3) +4(  1,2,5,6) ( 3,4)-4(  1,3,5,6) (4, 
   2) -4(1,4, 2,  5,6)+4(1,3, 4,   5,6)$$$$+4(3, 4, 2)(1,5,6)+4(1,2, 3,  5,6)+4(2,4)(1,5,6)-4 ( 3,4)(1,5,6)-4 (2,3) (1,5,6) .
$$
So expanding the terms containing 4  and 5-cycles: we finally have 

\bigskip

\scriptsize
 \begin{equation} \label{sei1} 8(1,2,3,4,5,6) \stackrel{\eqref{ilc}}=8(1,5,6)(1,2,3,4)=\end{equation}

   $ 5 -2(1,2)-3(1,3)+(1,6)+(2,3)-2(1,4)(2,3)-3(1,6)(2,3)-3(2,4)+2(1,3)(2,4)-(1,3)(2,6)-(1,4)(2,6)+(3,4)-2(1,2)(3,4)-3(1,6)(3,4)-2(2,5)(3,4)+2(1,6)(2,5)(3,4)-2(1,5)(2,6)(3,4)-2(3,5)+2(2,4)(3,5)-2(1,6)(2,4)(3,5)-(3,6)+(1,2)(3,6)-(1,4)(3,6)-2(1,5)(3,6)+2(1,5)(2,4)(3,6)+(1,6)(4,2)-4(1,6)(4,5)-2(2,3)(4,5)+2(1,6)(2,3)(4,5)-2(4,6)+(1,2)(4,6)+(1,3)(4,6)-2(1,5)(2,3)(4,6)-2(5,6)+2(1,3)(5,6)-2(2,3)(5,6)+2(1,4)(2,3)(5,6)+2(2,4)(5,6)-2(1,3)(2,4)(5,6)-2(3,4)(5,6)+2(1,2)(3,4)(5,6)+(1,2,3)+(1,2,4)+2(3,4)(1,2,5)-2(1,2,6)+2(3,4)(1,2,6)+2(3,5)(1,2,6)+2(4,5)(1,2,6)+(1,3,4)+2(1,3,5)-2(2,4)(1,3,5)+2(1,3,6)-2(2,4)(1,3,6)-2(2,5)(1,3,6)+2(4,5)(1,3,6)+2(2,3)(1,4,5)+2(1,4,6)+2(2,3)(1,4,6)-2(2,5)(1,4,6)-2(3,5)(1,4,6)+2(1,6)(2,3,5)+(2,3,6)+2(1,6)(2,4,5)+(2,4,6)+2(3,4)(2,5,6)+2(1,6)(3,4,5)+(3,4,6)+2(3,5,6)-2(2,4)(3,5,6)+2(2,3)(4,5,6)+4(1,5,6)(2,3,4)$
\bigskip
 
\normalsize
     Hence the final formula for the permutation of type 3,3 in term of special elements obtained by the method explained in section \ref{ur} is:
\scriptsize     
\begin{equation}\label{330}
8(4,3,2)(5,6,1)=
\end{equation}
   $(1,2)-2(1,3)+(1,5)+4(1,6)-3(2,3)+2(1,5)(2,3)-2(1,6)(2,3)-3(2,4)+2(1,5)(2,4)-2(1,6)(2,4)+7(2,5)-2(1,3)(2,5)-4(1,4)(2,5)-6(1,6)(2,5)-2(2,6)-2(1,3)(2,6)+4(1,4)(2,6)+2(1,5)(2,6)+(3,4)-2(1,5)(3,4)-2(1,6)(3,4)-4(2,5)(3,4)-4(3,5)+4(1,4)(3,5)+4(2,4)(3,5)+4(2,6)(3,5)-4(1,4)(2,6)(3,5)-4(3,6)+4(1,2)(3,6)-4(1,4)(3,6)-4(2,5)(3,6)+4(1,4)(2,5)(3,6)-2(4,5)+2(1,2)(4,5)+2(1,3)(4,5)+4(2,3)(4,5)+4(1,3)(2,6)(4,5)+4(3,6)(4,5)-4(1,2)(3,6)(4,5)-2(1,2)(4,6)+2(1,3)(4,6)-4(1,3)(2,5)(4,6)+4(1,2)(3,5)(4,6)+6(5,6)-4(1,2)(5,6)-4(4,5)(1,2,3)+4(5,6)(1,2,3)+(1,2,5)-4(3,6)(1,2,5)-2(1,2,6)+4(4,5)(1,2,6)+2(1,3,6)+4(2,5)(1,3,6)-4(4,5)(1,3,6)-4(3,5)(1,4,2)+4(5,6)(1,4,2)+4(2,5)(1,4,3)-4(5,6)(1,4,3)-4(2,6)(1,4,5)+4(3,6)(1,4,5)-(1,5,2)-4(1,5,6)+8(3,4)(1,5,6)+2(1,6,3)-2(1,6,5)+(2,3,5)+2(2,3,6)-4(4,5)(2,3,6)+4(1,6)(2,4,3)-(2,4,5)+4(1,6)(2,4,5)-(2,5,3)+4(1,6)(2,5,3)+(2,5,4)-2(2,5,6)+2(1,4)(2,5,6)+2(2,6,4)-4(3,5)(2,6,4)-2(1,4)(2,6,5)+(3,4,5)-4(1,6)(3,4,5)+2(3,4,6)-(3,5,4)+2(1,2)(3,5,6)-2(1,4)(3,5,6)+4(2,5)(3,6,4)-2(1,2)(3,6,5)+2(1,4)(3,6,5)-2(4,5,6)-2(4,6,5)$
   \normalsize

 \subsection{  The corresponding  rules of substitution:}\quad\bigskip
   
In writing an algorithm  to transform any given permutation into a sum of special  elements one thus uses the following {\em substitutional rules}. From:
\begin{equation}\label{q1}2(a,b, c, d)=\end{equation} 
$$\!\! \!\! \!\! \!\!  \!\! \!\!  (a,d )(b,c) +(a,b) ( c,d)-(a,c ) (d, 
   b) -(a,d, b )+(a,c, d)+(b,c, d)+(a,b, c)+(b,d)- ( c,d)- (b,c) .$$
and:\begin{equation}\label{id1}
( a,B)(a,A)= (a,A,B),\quad e.g.\  (1,2,3)(1,5,4,6)=(1,5,4,6,2,3)
\end{equation}
   for a cycle $\mathtt c :=(a,b,c,d,A)$ of length $n+4$, where $A$ is of length $n>0$, we have \begin{equation}\label{ilc1}
\mathtt c :=(a,b,c,d,A)=(a,A)(a,b,c,d)=  
\end{equation}
$$\!\! \!\! \!\! \!\!  \!\! \!\!  (a,d ,A)(b,c) +(a,b,A) ( c,d)-(a,c ,A) (d, 
   b) -(a,d, b ,A)$$$$+(a,c, d,A)+(b,c, d)+(a,b, c,A)+(b,d)- ( c,d)- (b,c) .$$
   
   This formula now contains only cycles  of length $<n+4$.\smallskip

By applying recursively these rules we arrive to a linear combination  in which no cycles of length $>3$ appear. Now the only reduction to be made is if  there are  pairs of 3-cycles. For these we finally  repeat the final rule:       
 \begin{equation}\label{33}
8(d,c,b)(e,f,a)=
\end{equation}$(a,b)-2(a,c)+(a,e)+4(a,f)-3(b,c)+2(a,e)(b,c)-2(a,f)(b,c)-3(b,d)+2(a,e)(b,d)-2(a,f)(b,d)+7(b,e)-2(a,c)(b,e)-4(a,d)(b,e)-6(a,f)(b,e)-2(b,f)-2(a,c)(b,f)+4(a,d)(b,f)+2(a,e)(b,f)+(c,d)-2(a,e)(c,d)-2(a,f)(c,d)-4(b,e)(c,d)-4(c,e)+4(a,d)(c,e)+4(b,d)(c,e)+4(b,f)(c,e)-4(a,d)(b,f)(c,e)-4(c,f)+4(a,b)(c,f)-4(a,d)(c,f)-4(b,e)(c,f)+4(a,d)(b,e)(c,f)-2(d,e)+2(a,b)(d,e)+2(a,c)(d,e)+4(b,c)(d,e)+4(a,c)(b,f)(d,e)+4(c,f)(d,e)-4(a,b)(c,f)(d,e)-2(a,b)(d,f)+2(a,c)(d,f)-4(a,c)(b,e)(d,f)+4(a,b)(c,e)(d,f)+6(e,f)-4(a,b)(e,f)-4(d,e)(a,b,c)+4(e,f)(a,b,c)+(a,b,e)-4(c,f)(a,b,e)-2(a,b,f)+4(d,e)(a,b,f)+2(a,c,f)+4(b,e)(a,c,f)-4(d,e)(a,c,f)-4(c,e)(a,d,b)+4(e,f)(a,d,b)+4(b,e)(a,d,c)-4(e,f)(a,d,c)-4(b,f)(a,d,e)+4(c,f)(a,d,e)-(a,e,b)-4(a,e,f)+8(c,d)(a,e,f)+2(a,f,c)-2(a,f,e)+(b,c,e)+2(b,c,f)-4(d,e)(b,c,f)+4(a,f)(b,d,c)-(b,d,e)+4(a,f)(b,d,e)-(b,e,c)+4(a,f)(b,e,c)+(b,e,d)-2(b,e,f)+2(a,d)(b,e,f)+2(b,f,d)-4(c,e)(b,f,d)-2(a,d)(b,f,e)+(c,d,e)-4(a,f)(c,d,e)+2(c,d,f)-(c,e,d)+2(a,b)(c,e,f)-2(a,d)(c,e,f)+4(b,e)(c,f,d)-2(a,b)(c,f,e)+2(a,d)(c,f,e)-2(d,e,f)-2(d,f,e)$ 
\smallskip

\noindent until we arrive at a linear combination of special elements.
 \bibliographystyle{amsalpha}
 \end{document}